\documentclass[12pt,reqno]{amsart}
\usepackage{cases}
\usepackage{amsthm,amsmath,amsfonts, amssymb,url,cite,color}
\usepackage{tikz}
\usepackage{tikzscale}
\usepackage{graphicx}
\usepackage{graphics}
\usepackage{pict2e}
\usepackage{collectbox}
\usepackage{cite}
\newtheorem{thm}{Theorem}
\newtheorem{lem}{Lemma}
\newtheorem{prop}[thm]{Proposition}

\newtheorem{remark}{Remark}

\def\dda{\rm{dda}}
\def\pk{\rm{pk}}

\def\pk{\rm {pk}}
\def\val{\rm {val}}

\def\da{\rm{da}}
\def\dd{\rm {dd}}
\def\pk{\rm {pk}}

\def\Z{\mathbb{Z}}

\def\S{{\mathfrak S}}
\def\B{\mathcal{B}}

\newcommand{\bea}{\begin{align}} 
\newcommand{\ena}{\end{align}}

\begin{document}
\title[A new combinatorial formula for  alternating descent polynomials]{A new combinatorial formula for alternating descent polynomials}
~\thanks{\today}
\author{Qiongqiong Pan }
\address{College of Mathematics and Physics, Wenzhou University\\
Wenzhou 325035, PR China}
\email{qpan@wzu.edu.cn}
\begin{abstract}  
We prove a combinatorial formula for the alternating 
descent polynomials of type A and B. Combining with Josuat-Vergès' combinatorial interpretation for Hoffman's derivative polynomials for tangent and secant functions,  we obtain
a unified combinatorial proof of Ma-Yeh's two 
 formulae linking these two families of polynomials.
  
%
\end{abstract}
\maketitle 
 \vskip 0.5cm 

\section{Introduction}


Let $\S_n$ denote the symmetric group of all permutations of $[n]:=\{1,2,\ldots,n\}$. For a permutation  $\sigma=\sigma(1)\sigma(2)\ldots\sigma(n)\in\S_n$,
 an index $i\in[n-1]$ is a  {\it{descent}} if  $\sigma(i)>\sigma(i+1)$. The generating function of 
descent statistic is the classical Eulerian polynomials.
In 2008  Chebikin~\cite{Che08} studied new statistics on permutations that are variations of the descent statistic.  We say that $\sigma=\sigma(1)\sigma(2)\cdots\sigma(n)$ has an \textit{alternating descent} at position $i$ if either $\sigma(i)>\sigma(i+1)$ and $i$ is odd, or else if $\sigma(i)<\sigma(i+1)$ and $i$ is even.

Let $\hat{D}(\sigma)$ be the set of  \textit{alternating descent positions} 
of $\sigma$, and set $\hat{d}(\sigma)=|\hat{D}(\sigma)|$. We define 
 the alternating Eulerian polynomials $\hat{A}(x)$
 by
\begin{align}
\hat{A}_n(x)=\sum_{\sigma\in\S_n}x^{\hat{d}(\sigma)}=\sum_{k=0}^{n-1}\hat{A}(n,k)x^k.
\end{align}

Denote by $\mathcal{B}_n$ the collection of permutations $\sigma$  of the set 
$[\pm n]:=\{ 1, 2,\ldots, n\}\cup \{ -1, -2,\ldots, -n\}$ such that $\sigma(-i)=\sigma(i)$ for all $i\in [n]$. Obviously, the word $|\sigma|:=|\sigma(1)|\cdots|\sigma(n)|$ is a permutation
in $\S_n$. An element of $\mathcal{B}_n$ is called a permutation of type $B$.
As usual, we always assume that type B permutations are prepended by $0$. 
That is, we identify an element $\sigma=\sigma(1)\cdots\sigma(n)$ in $\mathcal{B}_n$ with the word $\sigma(0)\sigma(1)\cdots\sigma(n)$, where $\sigma(0)=0$. 
Following Ma and Yeh~\cite{MY16}
we say that $\sigma\in\mathcal{B}_n$ has an alternating descent at position $i\in\{0\}\cup[n-1]$
if either $\sigma(i)<\sigma(i+1)$ and $i$ is even, or else if $\sigma(i)>\sigma(i+1)$ and $i$ is odd.

Let $\hat{D}_B(\sigma)$ be the set of positions at which $\sigma$ has an alternating descent, and set $\hat{d}_B(\sigma)=|\hat{D}_B(\sigma)|$. 
Similarly, we define the type B alternating Eulerian polynomials  $\hat{B}_n(x)$ by
\begin{align}
\hat{B}_n(x)=\sum_{\sigma\in\mathcal{B}_n}x^{\hat{d}_{B}(\sigma)}=\sum_{k=0}^{n}\hat{B}(n,k)x^k.
\end{align}

In what follows, to  any  $\sigma=\sigma(1)\ldots\sigma(n)\in\S_n$ we append 
 $\sigma(n+1)$ and prepend  $\sigma(0)$ in two different ways:
\begin{enumerate}
\item  $\sigma(0)=n+1$ and $\sigma(n+1)=n+1$; 
\item $\sigma(0)=0$ and $\sigma(n+1)=n+1$.
\end{enumerate}
Let $\S_n$ be the set of permutations  in $\S_n$ with the first convention and 
$\S_n^0$ the second convention, respectively. 
For  $\sigma\in\S_n\cup\S_n^0$, an integer $i\in[n]$ is 
\begin{itemize}
\item
a double ascent of $\sigma$ if $\sigma(i-1)<\sigma(i)<\sigma(i+1)$, 
\item a double descent of $\sigma$ if $\sigma(i-1)>\sigma(i)>\sigma(i+1)$, 
\item a valley of $\sigma$ if $\sigma(i-1)>\sigma(i)<\sigma(i+1)$,
\item a peak of $\sigma$ if $\sigma(i-1)<\sigma(i)>\sigma(i+1)$.
\end{itemize}
Let $\da(\sigma)$, $\mathrm{dd}(\sigma)$, $\val(\sigma)$ and  $\pk(\sigma)$ denote the number of double ascents,  double descents,
 valleys and  peaks of $\sigma$,  respectively. And let $\mathrm{DA}(\sigma)$, $\mathrm{DD}(\sigma)$, $\mathrm{VAL}(\sigma)$ and  $\mathrm{PEAK}(\sigma)$ denote the set of double ascents,  double descents,
 valleys and  peaks of $\sigma$,  respectively. It is easy to verify the
following facts:
\begin{align}\label{L1}
\forall \sigma\in\S_n, \quad \val(\sigma)&=\pk(\sigma)+1,\quad
2\,\val(\sigma)+\dda(\sigma)=n+1,\\
\forall \sigma\in\S_n^0, \quad 
\val(\sigma)&=\pk(\sigma), \quad 2\,\val(\sigma)+\dda(\sigma)=n,\label{L2}
\end{align}
where $\dda(\sigma)=\da(\sigma)+\dd(\sigma)$.
Our starting point is  the following result.
\begin{thm}\label{main1} 
We have 
\begin{align}
2^n\hat{A}_n(x)&=\sum_{\sigma\in\S_n}(1+x)^{\dda(\sigma)}2^{\val(\sigma)}(1+x^2)^{\pk(\sigma)},\label{pan1}\\
\hat{B}_n(x)&=\sum_{\sigma\in\S_n^0}(1+x)^{\dda(\sigma)}2^{\val(\sigma)}(1+x^2)^{\pk(\sigma)}.\label{pan2}
\end{align}
\end{thm}

Hoffman~\cite{Hoff99} considered  the derivative 
  polynomials $P_n(x)$ and $Q_n(x)$ appearing in  the successive derivatives of $\tan x$ and $\sec x$:
\begin{align*}
\frac{\textrm{d}^n}{\textrm{d}x^n}\tan(x)=P_n(\tan(x)),\quad \frac{\textrm{d}^n}{\textrm{d}x^n}\sec(x)=Q_n(\tan(x))\sec(x).
\end{align*}
Clearly these polynomials satisfy the differential equations
\begin{align}\label{p10}
P_{n+1}(x)=(1+x^2)P_n^{'}(x), \quad Q_{n+1}(x)=xQ_n(x)+(1+x^2)Q_n^{'}(x).
\end{align}
Hoffman  proved the following generating function formulas:
\begin{align}
\sum_{n\geq0}P_n(x)\frac{z^n}{n!}&=\frac{\sin(z)+xcos(z)}{\cos(z)-x\sin(z)},\label{hof1}\\
\sum_{n\geq0}Q_n(x)\frac{z^n}{n!}&=\frac{1}{\cos(z)-x\sin(z)}.\label{hof2}
\end{align}
For  $n\geq1$, the following combinatorial interpretation are due to Josuat-Vergès~\cite{JV14}.
 \begin{align}
P_n(x)&=\sum_{\sigma\in\S_n}x^{\dda(\sigma)}(1+x^2)^{\val(\sigma)},\\ Q_n(x)&=\sum_{\sigma\in\S_n^0}x^{\dda(\sigma)}(1+x^2)^{\val(\sigma)}.
\end{align}
By \eqref{L1} and \eqref{L2}, we derive from the above identities that 
\begin{align}
(1-x)^{n+1}P_n\left(\frac{1+x}{1-x}\right)
&=\sum_{\sigma\in\S_n}(1+x)^{\dda(\sigma)}(2+2x^2)^{\val(\sigma)}\nonumber\\
&=(1+x^2)\sum_{\sigma\in\S_n}(1+x)^{\dda(\sigma)}2^{\val(\sigma)}(1+x^2)^{\pk(\sigma)},\label{A1}
\end{align}
and
\begin{align}
(1-x)^{n}Q_n\left(\frac{1+x}{1-x}\right)
&=\sum_{\sigma\in\S_n^0}(1+x)^{\dda(\sigma)}(2+2x^2)^{\val(\sigma)}\nonumber\\
&=\sum_{\sigma\in\S_n^0}(1+x)^{\dda(\sigma)}2^{\val(\sigma)}(1+x^2)^{\pk(\sigma)}.\label{A2}
\end{align}
Combining Theorem~\ref{main1} with \eqref{A1} and \eqref{A2} we derive  the following relation linking the derivative polynomials and descent polynomials.
\begin{thm}\label{ma-yeh} We have 
\begin{align}\label{my1}
2^n(1+x^2)\hat{A}_n(x)=&(1-x)^{n+1}P_n\left(\frac{1+x}{1-x}\right),\\
\hat{B}_n(x)=&(1-x)^nQ_n\left(\frac{1+x}{1-x}\right).\label{my2}
\end{align}
\end{thm}
Now, invoking the generating functions  \eqref{hof1} and \eqref{hof2}, Theorem~\ref{ma-yeh} implies  immediately 
 the generating functions of the descent polynomials.
\begin{thm}\label{theorem1}
We have
\begin{align}\label{che}
\sum_{n>0}\hat{A}_n(x)\frac{z^n}{n!}&=\frac{\sec(1-x)z+\tan(1-x)z-1}{1-x(\sec(1-x)z+\tan(1-x)z)},\\
\label{p6}
\sum_{n\geq0}\hat{B}_n(x)\frac{z^n}{n!}&=\frac{x-1}{(x-1)\cos(z(1-x))+(x+1)\sin(z(1-x))}.
\end{align}
\end{thm}

Ma and Yeh~\cite{MY16} proved that  \eqref{my1} and \eqref{che} are equivalent and  conjectured \eqref{my2}. Formula \eqref{che} is due to 
Chebikin~\cite{Che08} 
and further generalized by Remmel~\cite{Re12} and Gessel and Zhuang~\cite{GZ}.
Recently, several \emph{classical approaches} to
\eqref{che} and \eqref{p6} were given by Ma-Fang-Mansour-Yeh~\cite{MA2022},
Lin-Ma-Wang-Wang~\cite[Theorem 2.1]{L-M-W-W} and   Ding-Zhu~\cite{D-Z}. Here, by \emph{classical approach} we mean that the proof consists of first establishes a linear recurrence relation of the polynomials, then converting the latter into a differential equation of the generating function, and finally solving the differential equation.

%

%
%

We shall  prove Theorem~\ref{main1} 
in section~2 and  give another proof of~\eqref{p6} in Section 3, which is inspired  by
Chebikin's proof of \eqref{che}  in \cite{Che08}.

\section{Proof of Theorem~\ref{main1}}

For a $\sigma\in\S_n$, we can obtain $2^n$ of type B permutations by signing $\sigma(i)$ with $\epsilon_i\in\{-1,1\}$, for $i\in[n]$.
Since sign $\epsilon_i$'s can be seen as involutions and that they commute, hence the group $\mathbb{Z}_2^n$ acts on $\S_n$ via the functions $\prod_{i\in S}\epsilon_i$, $S\subset [n]$. For $\sigma\in\S_n$, let $\mathrm{Orb}(\sigma)=\{g(\sigma):g\in\mathbb{Z}_2^n\}$ be the orbit of $\sigma$ under the signed action. For example, the orbit of $123$ in $\B_3$   has the following elements:
$$
\mathrm{Orb}(123)=\{123,\; \bar 1 23,\; 1\bar 2 3,\;  \;12\bar 3,\;
\bar 1\bar 23,\; 1\bar 2\bar 3,\;  \bar 12 \bar 3,\;
\bar 1\bar 2\bar 3\}.
$$

For convention, we still use $\hat{d}(\pi)$ to denote the number of alternating descent of $\pi\in\mathcal{B}_n$, the only difference of $\hat{d}(\pi)$ from $\hat{d}_B(\pi)$ is that we do not consider the 0 position in $\hat{d}(\pi)$, i.e., 
\begin{subnumcases}
{\hat{d}(\pi)=}\nonumber
\hat{d}_B(\pi)-1,& if\,\,$\pi(1)>0$\\\nonumber
\hat{d}_B(\pi),&if\,\,$\pi(1)<0$.
\end{subnumcases}

Let $E=\{a_1, \ldots, a_n\}$ be a subtset of $\Z$. A permutation (or list)
$\pi$ of elements of $E$  is order isomorphic to $\sigma\in \S_n$ if $\sigma$ is the permutation obtained from $\pi$ by replacing its $i$-th smallest entry with $i$.  For any $\sigma\in \S_n$, let $\mathrm{Iso}(\sigma)$ be the set of 
permutations isomorphic to 
$\sigma$ in $\B_n$. For example,  the permutations isormorphic to $123$ in $\B_3$ are
$$
\mathrm{Iso}(123)=\{123,\; \bar 1 23,\; \bar 3\bar2\bar 1, \;\bar 3 \bar 2 1,\; \bar 2 1 3,\;  \bar 2\bar 1 3,\;\bar 3 1 2,\;
 \bar 3\bar 1 2\}.
$$
Clearly we have
\begin{equation}\label{eqL}
2^n\hat{A}_n(x)=
\sum_{\sigma\in\S_n}\sum_{\pi\in \mathrm{Iso}(\sigma)}
x^{\hat{d}(\pi)}=\sum_{\pi\in\mathcal{B}_n}x^{\hat{d}(\pi)}.
\end{equation}


\begin{lem}\label{proposition1}
For $\sigma\in\S_n$ and $\pi=\epsilon_1\sigma(1)\epsilon_2\sigma(2)\cdots\epsilon_n\sigma(n)\in\mathrm{Orb}(\sigma)$.
\begin{itemize}
\item [(i)]
If $i$ is a double ascent in $\sigma$, then whether $i-1$ is an alternating descent position in $\pi$ or not, it is only determined by the sign of $\pi(i)$.
\item[(ii)]
If $i$ is a double descent in $\sigma$, then whether $i$ is an alternating descent position in $\pi$ or not, it is only determined by the sign of $\pi(i)$.
\item[(iii)]
If $i$ is a peak, then whether both $i-1$ and $i$ are alternating descents positions or neither $i-1$ nor $i$ is alternating descent position, it is only determined by the sign of $\pi(i)$.
\item[(iv)]
If $i$ is a valley, then whether $i$ is an alternating descent or not, it is only determined by the sign of $\pi(i+1)$.
\end{itemize}
\end{lem}
\begin{proof} 
For $\sigma\in\S_n$, let $\pi=\epsilon_1\sigma(1)\epsilon_2\sigma(2)\cdots\epsilon_n\sigma(n)\in\mathrm{Orb}(\sigma)$, if $i\in[n-1]$ and $\sigma(i-1)<\sigma(i)$, then we have
\begin{subnumcases}
{}\nonumber
\epsilon_{i-1}\sigma(i-1)<\epsilon_i\sigma(i),& if\,\,$\epsilon_i=1$\\\nonumber
\epsilon_{i-1}\sigma(i-1)>\epsilon_i\sigma(i),&if\,\,$\epsilon_i=-1$.
\end{subnumcases}
So, whether $i-1$ is an alternating descent position in $\pi$ or not, it is only determined by the sign of $\pi(i)$.
Similarly, if $\sigma(i-1)>\sigma(i)$, whether $i-1$ is an alternating descent position in $\pi$ or not, it is only determined by the sign of $\pi(i-1)$.

The properties (i)-(iv)  can be verified easily 
by the above oberservations.
\end{proof}
Then we prove the following Lemma.
\begin{lem}\label{thm1}
We have 
\begin{align}\label{Qi1}
\sum_{\pi\in\mathrm{Orb}(\sigma)}x^{\hat{d}(\pi)}=(1+x)^{\dda(\sigma)}2^{\val(\sigma)}(1+x^2)^{\pk(\sigma)}.
\end{align}
\end{lem}
\begin{proof}
In Lemma~\ref{proposition1}, for the case (iv), if $i$ is a valley of $\sigma$, then $i+1$ must be a peak or a double ascent in $\sigma$, so for position $i$, whether it can be an alternating descent or not should be considered in (i) or (iii) case, which means for $\pi\in\mathrm{Orb}(\sigma)$, the number of alternating descent of $\pi$ is only determined by the signs of peaks, double ascents and double descents of $\sigma$. That is, if sign a peak of $\sigma$ with $\epsilon_i\in\{-1,1\}$, then it will contribute either 0 or 2 alternating descents, i.e., $(1+x^2)$. If sign a double descent or a double ascent with $\epsilon_i\in\{-1,1\}$, then it will contribute either 0 or 1 alternating descent, i.e., $(1+x)$. For valleys of $\sigma$, each valley has two possibilities to sign it. Then we have~\eqref{Qi1}.
\end{proof}
Since
\begin{align}
\sum_{\pi\in\mathcal{B}_n}x^{\hat{d}(\pi)}=\sum_{\sigma\in\S_n}\sum_{\pi\in\mathrm{Orb}(\sigma)}x^{\hat{d}(\pi)},
\end{align}
by Lemma~\ref{thm1} we have~\eqref{pan1}.
Similarly, we prove~\eqref{pan2}.
\begin{remark}
In~\cite{JV14}, Josuat-Verg\`es considered a q-version of $Q_n(x)$ with a combinatorial interpretation in terms of cycle-alternating permutations, so it is natural to ask 
if  one can give a unified $q$-version for~\eqref{my1} and~\eqref{my2}.
\end{remark}
\section{Another proof of Theorem~\eqref{p6}}

Let $\sigma$ be  a permutation of type A or B. 
We say that $\sigma$  is 
an \textit{(down-up) alternating permutation}
with length $n$ if $\sigma(1)>\sigma(2)<\sigma(3)>\sigma(4)\cdots$. It is well known that the numbers of 
alternating permutations of type A are 
the Euler numbers $E_n$, which have  the  exponential generating function
\begin{align}\label{euler}
1+\sum_{n>0}E_n\frac{x^n}{n!}=\tan x+ \sec x.
\end{align}
 Let $\mathcal{DU}_n^{(B)}$ be the set of (down-up) alternating permutations of type B with length $n$.  Let $DU_n^{(B)}$ be the cardinality of $\mathcal{DU}_n^{(B)}$, clearly, we have $DU_n^{(B)}=2^nE_n$. Similarly we define 
the  \textit{up-down alternating permutations of type} A or B.
An alternating permutation of type B starting with a positive entry is called  a \textit{snake} (of type B). Let $\mathcal{S}_n$ be the set of snakes  of  
length $n$. Arnol'd~\cite{Arn92} showed that 
the  numbers    of snakes $S_n:=|\mathcal{S}_n|$   have  the generating function~\cite{Sp71},
\begin{align}\label{pan}
1+\sum_{n>0}S_n\frac{x^n}{n!}=\frac{1}{\cos(x)-\sin(x)}.
\end{align}

Let $\mathcal{B}_n^-=\{\sigma\in\mathcal{B}_n|\sigma(1)<0\}$, $\mathcal{B}_n^+=\{\sigma\in\mathcal{B}_n|\sigma(1)>0\}$, and  denote

\begin{align*}
\hat{B}_n^-(x)=\sum_{\sigma\in\mathcal{B}_n^-}x^{\hat{d}_B(\sigma)},\quad \hat{B}_n^+(x)=\sum_{\sigma\in\mathcal{B}_n^+}x^{\hat{d}_B(\sigma)},
\end{align*}
then $\mathcal{B}_n=\mathcal{B}_n^-\cup\mathcal{B}_n^+$ and $\hat{B}_n(x)=\hat{B}_n^-(x)+\hat{B}_n^+(x)$.

We begin by deducing a formula for the number of permutations in $\mathcal{B}_n^-$ with a given alternating descent set. For $\bar{S}\subseteq[n-1]$, let $\hat{\beta}_n^-(\bar{S})$ be the number of permutations $\sigma\in\mathcal{B}_n^-$ with $\hat{D}_B(\sigma)=\bar{S}$, and let $\hat{\alpha}_n^-(\bar{S})=\sum_{T\subseteq \bar{S}}\hat{\beta}_n^-(T)$ be the number of permutations $\sigma\in\mathcal{B}_n^-$ with $\hat{D}_B(\sigma)\subseteq \bar{S}$. For $\bar{S}=\{s_1<s_2<\cdots<s_k\}\subseteq[n-1]$, let $co(\bar{S})$ be the composition $(s_1,s_2-s_1,s_3-s_2,\ldots,s_k-s_{k-1},n-s_k)$ of $n$, and for a composition $\gamma=(\gamma_1,\ldots,\gamma_l)$ of $n$, let $\bar{S}_{\gamma}$ be the subset $\{\gamma_1,\gamma_1+\gamma_2\,\ldots,\gamma_1+\cdots+\gamma_{l-1}\}$ of $[n-1]$. Also, define
\begin{align*}
{n\choose \gamma}:={n\choose \gamma_1,\cdots,\gamma_l}=\frac{n!}{\gamma_1!\cdots\gamma_l!}.
\end{align*}

For $\sigma\in\mathcal{B}_n$, let $\sigma^-=\sigma^-(1)\ldots\sigma^-(n)\in\mathcal{B}_n$, where $\sigma^-(i)=-\sigma(i)$. Obviously, $\sigma\longmapsto\sigma^-$ is a bijection on $\mathcal{B}_n$.
Now, we have the following lemma.
\begin{lem}
We have 
\begin{align}\label{key}
\hat{\alpha}_n^-(\bar{S})={n\choose s_1, s_2-s_1,\cdots, n-s_k}\cdot S_{s_1}\cdot DU_{s_2-s_1}^{(B)}\cdots DU_{n-s_k}^{(B)}.
\end{align}
\end{lem}
\begin{proof}
Let $\bar{S}=\{s_1<s_2<\cdots<s_k\}\subseteq[n-1]$. Set $s_0=0$ and $s_{k+1}=n$ for convenience. The alternating descent set of a permutation $\sigma\in\mathcal{B}_n^-$ is contained in $\bar{S}$ if and only if for all $1< i\leq k+1$, the subword $\tau_i=\sigma(s_{i-1}+1)\sigma(s_{i-1}+2)\cdots\sigma(s_i)$ forms either an up-down(if $s_{i-1}$ is even) or a down-up(if $s_{i-1}$ is odd) alternating permutations of type B. And for the subword $\tau_1=\sigma(s_0+1)\sigma(s_0+2)\cdots\sigma(s_1)$,  $\tau^-_1$ forms a snake of type B. Thus to construct a permutation $\sigma$ with $\hat{D}_B(\sigma)\subseteq \bar{S}$, one must choose one of the ${n\choose co(\bar{S})}$ ways to distribute the elements of $[n]$ among the subwords $|\tau_1|,\ldots,|\tau_{k+1}|$, then for each $1<i\leq k$ choose one of the $\mathrm{DU}_{s_i-s_{i-1}}^{(B)} $ways of ordering the signed elements within the subword $\tau_i$ and for $i=1$ choose one of the $\mathrm{S}_{s_1}$ ways of ordering the signed elements within the subword $\tau^-_1$. Then the equation~\eqref{key} follows.
\end{proof}
Now, consider the sum
\begin{align}\label{p1}
\sum_{\bar{S}\subseteq[n-1]}\hat{\alpha}_n^-(\bar{S})\cdot x^{|\bar{S}|}=\sum_{\sigma\in\mathcal{B}^-_n}(\sum_{\hat{D}_B(\sigma)\subseteq T}x^{|T|}).
\end{align}
The right hand side of~\eqref{p1} is equal to 
\begin{align}
\sum_{\sigma\in\mathcal{B}_n^-}\sum_{T\supseteq\hat{D}_B(\sigma)}x^{\hat{d}_B(\sigma)+|T-\hat{D}_B(\sigma)|}&=\sum_{\sigma\in\mathcal{B}_n^-}x^{\hat{d}_B(\sigma)}\sum_{i=0}^{n-1-\hat{d}_B(\sigma)}{n-1-\hat{d}_B(\sigma)\choose i}x^i\nonumber\\
&=\sum_{\sigma\in\mathcal{B}_n^-}x^{\hat{d}_B(\sigma)}(1+x)^{n-1-\hat{d}_B(\sigma)},\label{p2}
\end{align}
as there are ${n-1-\hat{d}_B(\sigma)\choose i}$ subsets with $i$ elements  of $[n-1]\setminus\hat{D}_B(\sigma)$ containing $\hat{D}_B(\sigma)$.
Continuing with the right hand side of~\eqref{p2}, we get
\begin{align}\label{p3}
(1+x)^{n-1}\sum_{\sigma\in\mathcal{B}_n^-}\left(\frac{x}{1+x}\right)^{\hat{d}_B(\sigma)}=(1+x)^{n-1}\hat{B}_n^-\left(\frac{x}{1+x}\right).
\end{align}
Combining equations~\eqref{p1}-\eqref{p3}, we obtain
\begin{align}\label{p4}
\sum_{n\geq1}\left(\sum_{S\subseteq[n-1]}\hat{\alpha}_n(S)\cdot x^{|S|}\right)\frac{z^n}{n!}=\frac{1}{1+x}\sum_{n\geq1}\hat{B}_n^-\left(\frac{x}{1+x}\right)\cdot\frac{z^n(1+x)^n}{n!},
\end{align}
since $\bar{S}\longmapsto co(\bar{S})$ is a bijection between $[n-1]$ and the set of compositions of $n$, the left hand side of~\eqref{p4} is 
\begin{align*}
&\sum_{n\geq1}\left(\sum_{\gamma}\frac{S_{s_1}}{\gamma_1!}\cdot\frac{DU_{\gamma_2}^{(B)}}{\gamma_2!}\cdots\frac{DU_{\gamma_l}^{(B)}}{\gamma_l!}\cdot x^{l-1}\right)\cdot z^n\\
&=\sum_{l\geq 1}x^{l-1}\left(\sum_{i\geq 1}\frac{DU_i^{(B)}z^i}{i!}\right)^{l-1}\left(\sum_{i\geq1}\frac{S_iz^i}{i!}\right).
\end{align*}
As $DU_n^{(B)}=2^nE_n$, by \eqref{euler} 
and~\eqref{pan}, the last sum is equal to 
\begin{align*}
&\sum_{l\geq1}x^{l-1}(\tan(2z)+\sec(2z)-1)^{l-1}\cdot\left(\frac{1}{\cos(z)-\sin(z)}-1\right)\\
&=\left(\frac{1}{\cos(z)-\sin(z)}-1\right)\left(\frac{1}{1-x(\tan(2z)+\sec(2z)-1)}\right).
\end{align*}
thus, we have
\begin{align*}
\sum_{n\geq1}\hat{B}_n^-\left(\frac{x}{1+x}\right)\frac{[z(1+x)]^n}{n!}=
\frac{1+x}{1-x(\tan(2z)+\sec(2z)-1)}\left(\frac{1}{\cos(z)-\sin(z)}-1\right).
\end{align*}
\vskip 0.3cm
Replacing $z$ by $\frac{z}{1+x}$, $x$ by $\frac{x}{1-x}$, we obtain an explicit formula for the exponential generating function of $\hat{B}_n^-(x)$,

\begin{align}
\sum_{n\geq1}\hat{B}_n^-(x)\frac{z^n}{n!}=\frac{\cos(z(x-1))+\sin(z(x-1))-1}{(x-1)\cos(z(x-1))-(x+1)\sin(z(x-1))}.
\end{align}

Now, let $\hat{B}_0(x)=1$ then we can give an explicit formula for the exponential generating function of alternating Eulerian polynomials of type B.
%
\begin{proof}[Proof of Theorem \ref{theorem1}]
For $\sigma\in\hat{B}_n^-$, the bijection $\sigma\longmapsto\sigma^-$ from $\hat{B}_n^-$ to $\hat{B}_n^+$ defined by $\sigma^-(i)=-\sigma(i)$, shows that 
$$
\hat{B}_n^-(1/x)x^n=\hat{B}_n^+(x),
$$\vskip 0.3cm
\noindent  combine with $\hat{B}_n(x)=\hat{B}_n^+(x)+\hat{B}_n^-(x)$ and after some calculations, the equation~\eqref{p6} follows.
\end{proof}

%
\begin{prop}
For $n\geq1$, the numbers $\hat{B}(n,k)$ satisfy the recurrence relation,
\begin{align*}
\hat{B}(n+1,k)=&(n+1-k)(\hat{B}(n,k)+\hat{B}(n,k-2))
+k(\hat{B}(n,k+1)+\hat{B}(n,k-1))\\
&+\hat{B}(n,k+1)+\hat{B}(n,k-2),
\end{align*}
where $\hat{B}(n,k)=0$ if $k<0$ and $\hat{B}(n,0)=S_n$, see~\eqref{pan}.
\end{prop}
\begin{proof}
By \eqref{my2} and~\eqref{p10}, we have
\begin{align*}
\hat{B}_{n+1}(x)=(n+1+x+nx^2)\hat{B}_n(x)+(1-x)(1+x^2)\hat{B}_n^{'}(x).
\end{align*}
Then the recurrence follows through comparing the coefficients of two sides of the above equation.
\end{proof}
\begin{remark}
Remmel~\cite{Re12} considred  a different alternating descent of  type B permutations.
Lin-Ma-Wang-Wang ~\cite{L-M-W-W} also studied the positivity and divisibility  of alternating Eulerian polynomials of type A.
\end{remark}

%

\bibliographystyle{plain}
\bibliography{Alter}

\begin{thebibliography}{10}

\bibitem{Arn92}
Vladimir~I. Arnol{\textquotesingle}d.
\newblock The calculus of snakes and the combinatorics of bernoulli, euler and
  springer numbers of coxeter groups.
\newblock {\em Russian Mathematical Surveys}, 47(1):1--51, feb 1992.

\bibitem{Che08}
Denis Chebikin.
\newblock Variations on descents and inversions in permutations.
\newblock {\em Electronic Journal of Combinatorics}, 15(1):Research Paper 132,
  34, 2008.

\bibitem{D-Z}
Ming-Jian Ding and Bao-Xuan Zhu.
\newblock Stability of combinatorial polynomials and its applications.
\newblock {\em arXiv preprint arXiv:2016.12176}, 2021.

\bibitem{GZ}
Ira~M. Gessel and Yan Zhuang.
\newblock Counting permutations by alternating descents.
\newblock {\em The Electronic Journal of Combinatorics}, 21(4):P4--23, 2014.

\bibitem{Hoff99}
Michael~E. Hoffman.
\newblock Derivative polynomials, {E}uler polynomials, and associated integer
  sequences.
\newblock {\em Electronic Journal of Combinatorics}, 6:Research Paper 21, 13,
  1999.

\bibitem{JV14}
Matthieu Josuat-Verg\`es.
\newblock Enumeration of snakes and cycle-alternating permutations.
\newblock {\em The Australasian Journal of Combinatorics}, 60:279--305, 2014.

\bibitem{L-M-W-W}
Zhicong Lin, Shi-Mei Ma, David~GL Wang, and Liuquan Wang.
\newblock Positivity and divisibility of enumerators of alternating descents.
\newblock {\em The Ramanujan Journal}, pages 1--26, 2021.

\bibitem{MA2022}
Shi-Mei Ma, Qi~Fang, Toufik Mansour, and Yeong-Nan Yeh.
\newblock Alternating eulerian polynomials and left peak polynomials.
\newblock {\em Discrete Mathematics}, 345(3):112714, 2022.

\bibitem{MY16}
Shi-Mei Ma and Yeong-Nan Yeh.
\newblock Enumeration of permutations by number of alternating descents.
\newblock {\em Discrete Mathematics}, 339(4):1362--1367, 2016.

\bibitem{Re12}
Jeffrey~B. Remmel.
\newblock Generating functions for alternating descents and alternating major
  index.
\newblock {\em Annals of Combinatorics}, 16(3):625--650, 2012.

\bibitem{Sp71}
Tonny~A. Springer.
\newblock Remarks on a combinatorial problem.
\newblock {\em Nieuw Archief voor Wiskunde. Derde Serie (3)}, 19:30--36, 1971.

\end{thebibliography}

\end{document}